\documentclass[12pt]{amsart}

\usepackage{amssymb}

\usepackage{enumerate,color,titletoc}%,mathrsfs}
\usepackage[pagebackref,colorlinks,linkcolor=red,citecolor=blue,urlcolor=blue,hypertexnames=true]{hyperref}
\usepackage{url}

\renewcommand{\subset}{\subseteq}
\renewcommand{\emptyset}{\varnothing}

\def\bes{\begin{equation*} }
\def\ees{\end{equation*} }

\def\eps{\varepsilon}

\def\al{\alpha}
\def\be{\beta}
\def\ga{\gamma}

\makeatletter
\newsavebox\myboxA
\newsavebox\myboxB
\newlength\mylenA

\newcommand*\xoverline[2][0.75]{%
    \sbox{\myboxA}{$\m@th#2$}%
    \setbox\myboxB\null% Phantom box
    \ht\myboxB=\ht\myboxA%
    \dp\myboxB=\dp\myboxA%
    \wd\myboxB=#1\wd\myboxA% Scale phantom
    \sbox\myboxB{$\m@th\overline{\copy\myboxB}$}%  Overlined phantom
    \setlength\mylenA{\the\wd\myboxA}%   calc width diff
    \addtolength\mylenA{-\the\wd\myboxB}%
    \ifdim\wd\myboxB<\wd\myboxA%
       \rlap{\hskip 0.5\mylenA\usebox\myboxB}{\usebox\myboxA}%
    \else
        \hskip -0.5\mylenA\rlap{\usebox\myboxA}{\hskip 0.5\mylenA\usebox\myboxB}%
    \fi}
\makeatother

\newtheorem{theorem}            {Theorem}[section]
\newtheorem{thm}           [theorem]{Theorem}
\newtheorem{cor}          [theorem]{Corollary}
\newtheorem{prop}        [theorem]{Proposition}

\newtheorem{lem}              [theorem]{Lemma}

\newtheorem{ex}            [theorem]{Example}
\newtheorem{rem}             [theorem]{Remark}

\def\beq{ \begin{equation} }
\def\eeq{ \end{equation} }

\def\bep{\begin{proof}}
\def\eep{\end{proof}}

\def\ben{\begin{enumerate}}
\def\een{\end{enumerate}}

\def\bet{\begin{theorem}}
\def\eet{\end{theorem}}

\def\bel{\begin{lemma}}
\def\eel{\end{lemma}}

\def\la{\langle}

\def\mt{\nu}

\def\tr{\mbox{tr}}
\def\Ll{\Lambda}

\newcommand{\Mnns}{\R^{n\times n}}
\def\x{x}

\newcommand{\ax}{\langle\x\rangle}

\def\SRnn{\bbS_n}
\def\cB{ {\mathcal B} }

\def\cD{ {\mathcal D} }

\def\cH{ {\mathcal H} }

\def\cN{ {\mathcal N} }

\def\cS{{\mathcal S} }

\def\cX{\mathcal X}

\newcommand{\N}{{\mathbb N}}

\newcommand{\R}{{\mathbb R}}

\def\smatn{\bbS_n } 
\def\gtupn{\smatn^g}

\def\bbS{{\mathbb S}}

\def\ze{\zeta}

\def\La{\Lambda}

\def\del{\delta}

\def\la{\lambda}
\def\lL{\hat{\la}}

\def\ps{\mathfrak{P}}

\newcommand{\df}[1]{{\bf{#1}}{\index{#1}}}
\makeindex

\begin{document}

\title{The convex Positivstellensatz in a free algebra}

\author[Helton]{J. William Helton${}^1$}
\address{J. William Helton, Department of Mathematics\\
  University of California \\
  San Diego}
\email{helton@math.ucsd.edu}
\thanks{${}^1$Research supported by NSF grants
DMS-0700758, DMS-0757212, and the Ford Motor Co.}

\author[Klep]{Igor Klep${}^{2}$}
\address{Igor Klep, Department of Mathematics, 
The University of Auckland}
\email{igor.klep@auckland.ac.nz}
\thanks{${}^2$Research supported by the Slovenian Research Agency grants 
J1-3608 and P1-0222. Partly supported
by the Mathematisches Forschungsinstitut Oberwolfach Research in
Pairs RiP program. Partly supported by the program ``free spaces for creativity'' 
at the University of Konstanz.
The author thanks Markus Schweighofer for valuable
discussions. }

\author[McCullough]{Scott McCullough${}^3$}
\address{Scott McCullough, Department of Mathematics\\
  University of Florida, Gainesville %\\
   % Box 118105\\
   %  Gainesville, FL 32611-8105\\
   %  USA
   }
   \email{sam@math.ufl.edu}
\thanks{${}^3$Research supported by NSF grants DMS-0758306 and DMS-1101137.}

\date{\today}

\setcounter{tocdepth}{2}
\contentsmargin{2.55em} 
\dottedcontents{section}[3.8em]{}{2.3em}{.4pc} 
\dottedcontents{subsection}[6.1em]{}{3.2em}{.4pc}
%\dottedcontents{subsubsection}[8.4em]{}{4.1em}{.4pc}

\makeatletter
\newcommand{\mycontentsbox}{%
{\centerline{NOT FOR PUBLICATION}
\small\tableofcontents}}
\def\enddoc@text{\ifx\@empty\@translators \else\@settranslators\fi
\ifx\@empty\addresses \else\@setaddresses\fi
\newpage\mycontentsbox}
\makeatother

\begin{abstract}
  Given  a monic linear pencil $L$ in $g$ variables, 
  let $\ps_L=(\ps_L(n))_{n\in\N}$ where 
$$
 \ps_L(n):= \big\{ X\in \SRnn^g \mid L(X)\succeq0\big\},
$$
 and $\SRnn^g$ is the set of $g$-tuples of symmetric $n\times n$
 matrices.  Because $L$ is a monic linear pencil, each 
  $\ps_L(n)$ is convex  with interior, and conversely it is known that convex 
 bounded noncommutative semialgebraic sets with interior 
are all of the form $\ps_L$.
 The main result of this paper establishes a perfect 
 noncommutative Nichtnegativstellensatz on a convex semialgebraic set.
 Namely, a noncommutative matrix-valued polynomial 
 $p$ is \emph{positive semidefinite}
 on $\ps_L$ if and only if it
 has a \emph{weighted sum of squares representation} with \emph{optimal degree bounds}:
\[
 p  = s^* s   + \sum_j^{\rm finite} f_j^* L f_j,
\]
where
 $s, f_j$ are matrices of noncommutative polynomials of degree no greater than
 $\frac{\deg(p) }{2}$.
 This noncommutative result contrasts sharply with the commutative setting,
 where there is no control on the
 degrees of $s, f_j$
 and assuming only $p$ nonnegative,
as opposed to $p$ strictly positive,
   yields a clean
 Positivstellensatz so seldom that such cases are noteworthy.
\end{abstract}

%%%%%%%% below a text-only abstract
\iffalse
Given a monic linear pencil L in g variables let D_L be its positivity domain, i.e., the set of all g-tuples X of symmetric matrices of all sizes making L(X) positive semidefinite. Because L is a monic linear pencil, D_L is convex with interior, and conversely it is known that convex bounded noncommutative semialgebraic sets with interior are all of the form D_L. The main result of this paper establishes a perfect noncommutative Nichtnegativstellensatz on a convex semialgebraic set. Namely, a noncommutative polynomial p is positive semidefinite on D_L if and only if it has a weighted sum of squares representation with optimal degree bounds: p = s^* s + \sum_j f_j^* L f_j, where s, f_j are vectors of noncommutative polynomials of degree no greater than 1/2 deg(p). This noncommutative result contrasts sharply with the commutative setting, where there is no control on the degrees of s, f_j and assuming only p nonnegative, as opposed to p strictly positive, yields a clean Positivstellensatz so seldom that such cases are noteworthy.
\fi
%%%%%%%%%

\maketitle

\section{Introduction}
 A Positivstellensatz is an algebraic certificate for 
 a given polynomial $p$ to have a specific 
positivity property
 and such theorems date back in some form for over one hundred years
 for conventional (commutative) polynomials, cf.~\cite{BCR,Las,Lau,Mar,PD,Sce}.
 Positivstellens\"atze for  polynomials
 in noncommuting  variables are  creatures of this century - see
 \cite{HKM2,HM,KS1,PNA,DLTW}; for software equipped to dealing
with positive noncommutative polynomials we refer to \cite{HOSM,CKP}. 
 Often in the noncommutative setting 
 such theorems have cleaner statements than their commutative counterparts.
 For instance, a multivariate (commutative) polynomial on $\mathbb R^g$
 which is pointwise nonnegative need not be a sum of squares,
 but a  noncommutative polynomial
 which is nonnegative (in a sense made precise below) {\em is}
 a sum of squares - a result of the first author \cite{Hel}.

Classical commutative
Positivstellens\"atze  generally require $p$ to be strictly positive -   
 the cases where nonnegative  suffices are few and noteworthy, cf.~\cite{Sce},
 and the degrees of the  polynomials appearing in the representation of $p$
as a weighted sum of squares 
 are typically very high  compared to that of $p$. 
Furthermore, the semialgebraic set under consideration is often assumed
to be bounded \cite{Smu,Put}.

The main result of \cite{HM}
 gave a Positivstellensatz for matrix-valued noncommutative polynomials
 which was an exact extension, warts and all
(the strict  positivity
 assumption, possibility of high degree weights, and boundedness),
  of the commutative Putinar Positivstellensatz
 \cite{Put}.
 While gratifying, it was not, as in retrospect we have come to expect
 in the free algebra setting,  cleaner than its
  commutative counterpart. 
 What we find in this paper for noncommutative polynomials is
 that when the underlying  semialgebraic set
%, defined 
% by a matrix-valued noncommutative polynomial $q$, is
%  \emph{convex}, 
 is defined by a concave matrix-valued noncommutative polynomial $q$, %
 a {\it ``perfect" Positivstellensatz} holds; namely, 
 a representation 
 \bes
 p= \sum_j^{\rm finite} s_j^* s_j   + \sum_j^{\rm finite} f_j^* q f_j 
 \ees
 where $s_j, f_j$ are noncommutative matrix-valued 
 polynomials of degree no greater than
 $\frac{\deg(p) + 2}{2}$ holds
 for any $p$ which is ``nonnegative" 
 on the set $\ps_q$ where $q$ is ``nonnegative,"
 irrespective of the boundedness of the semialgebraic set
$\ps_q$
 defined by $q$.
% ** NEW although we require $\cD_q$  to be the closure
% of its interior. **
% In particular, the main result in this article says,
%  under the  hypothesis that the noncommutative
% polynomial  $q$ is concave, 
% if  $p$ nonnegative on the set
% $\cD_q$, the set where $q$ is nonnegative, then $p$
% has  a ``perfect" Positivstellensatz representation.
 Indeed this result is a Nichtnegativstellensatz, as $p$ 
 is only assumed to be nonnegative
 on $\ps_q$.  Thus, compared with the 
 main result of \cite{HM}, the hypothesis that $q$ is concave % convex 
 has been added, but the boundedness (or archimedean) hypothesis
 as well as the strict positivity hypothesis have been dropped,
 and the resulting  weighted sum of squares representation 
 is improved by giving optimal degree bounds. 
 As a corollary, when $q=1$ and $\ps_q$ is  everything,
   we recover the result  mentioned 
 in the first paragraph: nonnegative noncommutative polynomials are
 sums of squares.

 In the remainder of this 
 introduction, we state our main result
 after providing the needed background and
 definitions. 
 Then we give some examples.

\subsection{Words and NC polynomials}%\label{subsec:NCpoly}
 Given positive integers $n$ and $g$,
  let  $(\Mnns)^{g}$
  denote the set of $g$-tuples of real $n\times n$ matrices.
  A natural norm on $(\Mnns)^{g}$ is given by 
 \[
  \|X\|^2 =\sum^g \|X_j\|^2
 \]
  for $X=(X_1,\dots,X_g)\in (\Mnns)^{g}$. 
We use $\SRnn$ to denote real symmetric $n\times n$ matrices.

We write $\ax$ for the monoid freely
generated by $\x=(x_1,\ldots, x_g)$, i.e., $\ax$ consists of \df{words} in the $g$
noncommuting
letters $x_{1},\ldots,x_{g}$
(including the {\bf empty word $\emptyset$} which plays the role of the identity).
Let $\R\ax$ denote the associative
$\R$-algebra freely generated by $\x$, i.e., the elements of $\R\ax$
are polynomials in the noncommuting variables $\x$ with coefficients
in $\R$. Its elements are called \df{(nc) polynomials}.
An element of the form $aw$ where $0\neq a\in \R$ and
$w\in\ax$ is called a \df{monomial} and $a$ its
\df{coefficient}. Hence words are monomials whose coefficient is
$1$.
Endow $\R\ax$ with the natural \df{involution} ${}^*$
which fixes $\R\cup\{\x\}$ pointwise, reverses the
 order of words, and acts linearly on polynomials. 
For example, 
$(2-  3 x_{1}^2 x_{2} x_{3})^* =2  -3 x_{3} x_{2} x_{1}^2.$ 
Polynomials invariant with respect to this involution are \df{symmetric}.
The
length of the longest word in a noncommutative polynomial $f\in\R\ax$ is the
\df{degree} of $f$ and is denoted by $\deg( f)$. The set
of all words of degree at most $k$ is $\ax_k$, and $\R\ax_k$ is the vector
space of all noncommutative polynomials of degree at most $k$.

 Fix positive integers
 $\mt$ and $\ell$.  \df{Matrix-valued noncommutative polynomials} -- 
  elements of 
$\mathbb R^{\ell\times\mt}\ax=\R^{\ell\times\mt}\otimes \R\ax;$ 
 i.e.,
 $\ell\times\mt$ matrices with entries from $\R\ax$ -- will play a role
 in what follows. Elements of $\mathbb R^{\ell\times\mt}\ax$ are conveniently 
 represented using tensor products as
\beq\label{eq:preeq0}
  P=\sum_{w\in\ax} B_w \otimes w\in\R^{\ell\times\mt}\ax,
\eeq
  where $B_w\in \R^{\ell\times\mt}$, and the
 sum is finite. 
  Note that the involution ${}^*$ extends to matrix-valued polynomials by
\[
  P^* =\sum_w B_w^* \otimes w^*\in\R^{\mt\times\ell}\ax. 
\]
  If $\mt=\ell$ and $P^*=P$, we say $P$ is \df{symmetric}.

In the sequel, the tensor product will be reserved to 
denote the (Kronecker) tensor product
of matrices. 
Thus we will omit the tensor product notation for matrix-valued polynomials
and instead of \eqref{eq:preeq0} write simply
\[
 P=\sum_{w\in\ax} B_w  w\in\R^{\ell\times\mt}\ax.
\]

\subsubsection{Polynomial evaluations}
 If $p\in\R\ax$ is a noncommutative polynomial and $X\in (\Mnns)^{g}$,
 the evaluation $p(X)\in\Mnns$ is defined in the natural way by 
 replacing $x_{i}$ by $X_{i}$ and sending the empty word to
  the appropriately sized identity matrix. 

%%%%%%% explicit poly eval commented out
%For example, if $p=3 x_{1} x_{2}$, 
%then
%$$
%p\left( \begin{bmatrix}
%0&1 \\
%1 & 0
%\end{bmatrix},
%\begin{bmatrix}
%1&0\\
%0&-1
%\end{bmatrix}
%\right)=
%3  \begin{bmatrix}
%0&1 \\
%1 & 0
%\end{bmatrix}
%\,
%\begin{bmatrix}
%1&0\\
%0&-1
%\end{bmatrix}
%=
%\begin{bmatrix}
%0 & -3\\
%3 & 0 
% \end{bmatrix}.
%$$
%$Similarly, if $p(x)=\alpha\in\R$ and $X\in (\Mnns)^{g}$,
%$then $p(X)=\alpha I_n$.
%%%%  explict poly eval commented out

Most of our evaluations will be on tuples of \emph{symmetric}
matrices $X\in \SRnn^{g}$; our involution
fixes the variables $\x$ elementwise, so only these evaluations
give rise to $*$-representations of noncommutative polynomials.
Polynomial evaluations extend to matrix-valued polynomials by evaluating 
entrywise. Note that if $P\in\R^{\ell\times \ell}\ax$ is symmetric, and $X\in      \SRnn^{g}$, then
 $P(X)\in\R^{\ell n\times \ell n}$ is a symmetric matrix. 

\subsection{Linear and concave polynomials}
  If $A_1,\dots,A_g$ are symmetric $\ell \times \ell$ matrices, then
\beq
\label{eq:defLam}
  \Ll_A: = \sum_{j=1}^g A_j x_j
\eeq
  is a (homogeneous) symmetric  linear matrix-valued polynomial, also called a
\df{(homogeneous) linear pencil}. 
To $\Ll_A$ we associate the \df{monic linear pencil}
$$
I-\Ll_A = I_{\ell} - \sum_{j=1}^g A_j x_j.
$$ 

 A symmetric $q\in \mathbb R^{\ell\times \ell}\ax$ is \df{concave} provided
\bes
q\big(tX + (1-t)Y\big) \succeq t q(X) +(1-t)q(Y), \quad 0 \leq t \leq 1
\ees
for all $n\in\N$ and $X,Y \in \gtupn$.  
 The main result in \cite{HM0}  tells us that if 
  $q$ is scalar-valued (i.e., $\ell=1$)
 and $q(0)=I_\ell$, then $q$  is concave if and only if it has the form
\beq
\label{eq:quad}
q(x)= I_\ell - \Ll(x) - s^*(x) s(x)
\eeq
 for some homogeneous linear polynomial $\Ll\in\R\ax$ and homogeneous linear 
 vector-valued $s\in\R^{\ell\times 1}\ax$. 
  This result remains true, with the obvious modifications,
  for $q$ matrix-valued.  A proof is given
  in Subsection \ref{subsec:matrix-concave}.

\subsection{The Positivstellensatz}
For $f\in\R^{\ell\times\mt}\ax$, an element of the form $f^*f
\in \R^{\mt\times\mt}\ax$ will be called a
\df{(hermitian) square}.
Let 
$\Sigma^{\mt}$ denote the cone of sums of squares of
$\mt\times\mt$ matrix-valued polynomials, and, 
 given a nonnegative integer $N$, 
let $\Sigma_N^\mt\subseteq\Sigma^\mt$ denote sums of squares of 
polynomials of degree at most $N$.  Thus elements of $\Sigma_N^\mt$ 
have degree at most $2N$,
i.e., $\Sigma^\mt_N\subseteq\R^{\mt\times\mt}\ax_{2N}$.   
Conversely, since the highest order terms in a sum of squares
cannot cancel, we have $\R^{\mt\times\mt}\ax_{2N}\cap\Sigma^\mt=\Sigma^\mt_N$.

 Fix a symmetric $q\in \mathbb R^{\ell\times\ell}\ax$.
 Let
\[
 \ps_q(n) : =\{X \in \gtupn \mid   q(X) \succeq 0 \} 
\qquad \mbox{and} \qquad 
\ps_q := \bigcup_{n\in\N}  \ps_q(n).
\]

 Given $\alpha,\beta\in\N$, set 
\begin{equation}
 \label{eq:Malbeta}
  M_{\al, \beta}^{\mt}(q):=
   \Sigma_{\al}^{\mt}  +\Big\{ \sum_i^{\rm finite}  f_i^* qf_i \mid 
   \  f_i \in \R^{\ell\times \mt}\ax_\beta \Big\} 
  \  \subseteq \ \R^{\mt\times \mt}\ax_{\max \{2\al, 2\beta +a\}},
\end{equation}
 where $a=\deg(q)$.
Obviously, if $f\in M_{\al,\beta}^{\mt}(q)$ then $f|_{\ps_q}\succeq0$.

We call $M_{\al,\beta}^\mt(q)$ the \df{truncated quadratic module} and 
$\ps_q$ the \df{noncommutative (nc) semialgebraic set} defined by $q$.
If $q$ has degree one, then $\ps_q$ is also called an 
\df{LMI (linear matrix inequality) domain}.
 We often abbreviate $  M_{\al, \beta}^{\mt}(q)$   to 
$  M_{\al, \beta}^{\mt}$.   If $q(0)=I$ ($q$ is \df{monic}), then
 $\ps_q$ contains an \df{nc neighborhood of $0$};
 i.e., there exists $\eps>0$ such that for each $n\in\N$, if
  $X\in\SRnn^{g}$ and $\|X\|<\eps$,
  then $X\in \ps_q$.  
 Likewise $\ps_q$ is called \df{bounded} provided there is a number $R$
 for which all $X\in \ps_q$ satisfy $ \|X\| <R$.

 The following  is the free convex Positivstellensatz, the main
result of this paper.

\begin{thm}[Convex Positivstellensatz]
\label{thm:mainConcave}
 Suppose $q\in\R^{\ell\times \ell}\ax$ and $p\in\R^{\mt\times\mt}\ax$ are
symmetric 
  matrix-valued  noncommutative polynomials.
\ben[\rm (1)]
\item
 \label{thmmainitem1}
 If $q$ is  concave and monic 
 and $\deg(p)\leq 2d+1$,
  then 
\[
p(X)\succeq 0 \text{ for all } X\in \ps_q\quad\iff\quad
   p\in M_{d+1,d}^{\mt}(q).\] 
\item
 \label{thmmainitem2}
 If $q$ is a monic linear pencil
  and $\deg(p)\leq 2d+1$, then
\[
p(X)\succeq 0 \text{ for all } X\in \ps_q\quad\iff\quad
   p\in M_{d,d}^{\mt}(q).\] 
\een
If, in addition, the set $\ps_q$ is  bounded, 
the right-hand side of {\rm (1)}
is 
equivalent to
\[ p \in
\Big\{  \sum_j^{\rm finite} f_{j}^* q f_{j} \mid 
   f_{j}\in\R^{\ell\times \mt}\ax_{d+1} \Big\} = : \mathring M_{d+1}^\mt(q),
\]
while the right-hand side of {\rm (2)} is equivalent to
$p\in  \mathring M_{d}^\mt(q)$.
%%{\igor ** please check.  OK by bill}
\end{thm}

\begin{proof}
The proof of (1) and (2) is laid out in Subsection \ref{subsec:fromdplusonetod}.
The last fact is an immediate consequence of (1) and (2) and
Proposition \ref{prop:IVLV};
see Subsection \ref{subsec:dominate} for details.
\end{proof}

\begin{rem}\rm
% \label{rem:testrank}
  It is easy to see that given $k,\mt\in\N$ there exists a positive
  integer $t$ so that for 
  a symmetric $p\in\R^{\mt\times \mt}\ax_k$,
we have
  $p(X)\succeq 0$ for all $X\in \ps_q$ if and only
  if $p(X)\succeq 0$ for all $X\in\ps_q(t).$
  The smallest such $t$ 
  is called the $(k,\mt)$-\df{test rank} of $\ps_q$.
  Routine arguments show
  that this $(k,\mt)$-test rank is at most 
 $\mt \sigma_{\#}\big(\lceil \frac k2\rceil \big)$,
  where $$\sigma_{\#}(d):=\dim \R\ax_d= \sum_{j=0}^d g^j,$$
and $\lceil r \rceil$ denotes 
the smallest integer not less than $r$.

   There is also a 
   bound on the number of 
   summands in a certificate of the form $p\in M_{d+1,d}^{\mt}(q)$ or
$p\in M_{d,d}^{\mt}(q)$, 
   coming from Caratheodory's theorem
\cite[Theorem I.2.3]{Ba}
   on convex subsets of finite dimensional spaces.  
   For example, in case (1) of Theorem \ref{thm:mainConcave} it is $1+ \dim
\big( \R^{\mt\times\mt}\ax_{2d+1} \big)= 1+ \mt^2 \sigma_{\#}(2d+1)$.
\end{rem}

\begin{rem}\rm
 The main result of 
\cite{HM3} says that if $q$ is symmetric, matrix-valued,
  monic, and the connected component, $\cD_q,$ of $0$
 of
\[
\mathring \ps_q := \bigcup_{n\in\N} \big\{ X \in \gtupn \mid q(X)\succ 0 \big\}
\]
 is bounded and  convex, then there is a monic  linear
 pencil $L$ such that the closure of $\cD_q$ 
is of the form $\ps_L$. In particular, if $\mathring \ps_q$ 
 is itself convex, then its closure is $\ps_L$
 for some $L$.  %% \marginpar{\igor please check** OK by bill}  
 In this sense,\\[.1cm]
\centerline{\emph{Theorem {\rm\ref{thm:mainConcave}}
 establishes a perfect Positivstellensatz on a convex nc 
 semialgebraic set}.}  
\end{rem}

\begin{rem}\rm
% \label{rem:dominate}
In \cite{HKM} we studied LMI domains and their inclusions.
The linear Positivstellensatz there
\cite[Theorem 1.1]{HKM} states the following:
If $q,r$ are two monic linear pencils with 
$\ps_q$ bounded and $r$ is of size $\mt\times\mt,$ then
$\ps_q\subseteq \ps_r$ if and only if $r\in\mathring M^{\mt}_{0}(q)$.
So this is a very special case
of Theorem \ref{thm:mainConcave}.
Furthermore, \cite[Theorem 5.1]{HKM} is a very 
weak form of Theorem \ref{thm:mainConcave}.
 The techniques of proof in \cite{HKM}
are completely different than those here.
 We give further details and discuss the connection to complete positivity
 in Subsection \ref{subsec:dominate}. 
Intriguing is the fact that the special case of Theorem 
\ref{thm:mainConcave}
  where $p$ is affine linear implies a version of the Arveson Extension 
Theorem and the Stinespring
  Representation for matrices (as opposed  to operators).
\end{rem}

%\begin{rem}\rm
% \label{rem:needmonic}
The conclusion of Theorem \ref{thm:mainConcave} may fail
if $q$ is not assumed to be monic as the following
 examples show.
%\end{rem}

\begin{ex}\rm
%\ben[\rm(1)]
%\item
Let $$q=\begin{bmatrix}x & 1 \\ 1 & 0 \end{bmatrix} \in \R^{2\times 2}\ax_1.$$
Then $\ps_q=\varnothing$, so $p:=-1 \in \R^{1\times 1}\ax_0$ satisfies
$-1|_{\ps_q}\succeq0$, but $-1\not\in M^1_{0,0}$. 
However,
for $$u:=\begin{bmatrix}1 & -1-\frac x2\end{bmatrix}^*,$$ we have
$$-1=\frac12u^*qu,$$
showing that $-1\in\mathring M_{1}^1$.

For details and more on
the study of \emph{empty} LMI domains we refer the reader to \cite{KS}.
One of the main results there states that $\ps_q$ is empty (for a nonhomogeneous
linear pencil $q$) if and only if the truncated quadratic module 
$M_{\al,\al}^1(q)$ (in the ring $\R[x]$ of polynomials in \emph{commuting} variables) contains $-1$ for some (explicitly computable) $\al\in\N$.
\end{ex}

%\item
\begin{ex}\rm
For another example consider
$$q= \begin{bmatrix} 1 & x \\ x & 0\end{bmatrix}.$$
Then $\ps_q=\{0\}$. Hence obviously $x\succeq0$ on $\ps_q$. 
But it is easy to see that $x\not\in M_{\al,\be}^1(q)$ for any $\al,\be\in\N$; cf.~\cite[Example 2]{Za}.
%\een
\end{ex}

\def\mk{\kappa}
\subsection{Guide to the rest of the paper}

  Given $\alpha,\beta\in\N$, let  $a=\deg(q)$ and 
$$\mk=\max\{2\al,2\beta + a\}.$$

In view
  of Theorem \ref{thm:mainConcave},
  we say that the truncated quadratic module $M^{\mt}_{\al,\be}(q)$ has the 
  \df{$\theta$-PosSs-property}
  if, for a symmetric polynomial $p\in \R^{\mt\times\mt}\ax_\theta$,
  the property $p(X)\succeq 0$ for all $X\in\ps_q$ implies
  $p\in M^{\mt}_{\al,\be}(q)$ (the converse being automatic).  Note that
    $M^{\mt}_{\al,\be}(q) \subset \R^{\mt\times\mt}\ax_\theta$ and thus the
     definition is sensible only for $\theta \le \mk$.
  
  The difficult part in proving Theorem \ref{thm:mainConcave}
  is showing that $M^{\mt}_{d+1,d}(q)$ 
 has the $(2d+1)$-PosSs-property in the case that $q$ is a monic linear pencil.
  The argument occupies the bulk of this
 article. The reduction to this case
  and other preliminaries are in
  the following section, Section \ref{sec:red}.
  The passages from $q$ linear to $q$
 concave and from $M^{\mt}_{d+1,d}(q)$ to $M^{\mt}_{d,d}(q)$ are
  rather simple and the details are 
 found in Subsections
  \ref{subsec:concaveToLin} and \ref{subsec:fromdplusonetod}.
  Section \ref{sec:red} ends with a brief discussion of
  connections to Hankel matrices
  and free noncommutative moment problems.
 The proof of Theorem \ref{thm:mainConcave} 
 culminates in Subsection \ref{subsec:sep}, using the
 results on positive linear functionals from Subsection \ref{subsec:flathank}.

 In the last section we discuss 
 connections to LMI domination and complete positivity (Subsection 
 \ref{subsec:dominate}), and  outline in Subsection \ref{subsec:beyondCon}
 an improvement of the results of \cite{HMP} obtained by
 the approach here in the absence of concavity of $q$
  (or convexity of the underlying semialgebraic set).

% in the absence of concavity of $q$ the approach here
% generates an improvement on the results of \cite{HMP}.

\section{Reductions and preliminaries}
 \label{sec:red}

  In this section we make first steps towards the proof of 
  Theorem \ref{thm:mainConcave}. We start by giving 
 preliminaries on concave polynomials needed for two reductions
 in the subsequent subsections.

\subsection{Concave polynomials}
 \label{subsec:matrix-concave}
  The structure of  symmetric concave matrix-valued polynomials is 
  quite rigid.

\begin{prop}
 \label{prop:concaveform}
  If $q$ is a symmetric concave matrix-valued polynomial with $q(0)=I$,  
  then
  there exists a homogeneous linear pencil $\Lambda$ 
  and a homogeneous linear matrix-valued polynomial
  $s$ such that
\[
   q = I-\Lambda - s^* s. 
\]
\end{prop}

\begin{proof}
  Suppose $q$ is an $\ell\times \ell$ matrix-valued symmetric polynomial. Thus,
  using the tensor product notation, 
\[
  q=\sum_{w\in\ax} Q_w \otimes w,
\]
  for some $\ell\times \ell$ matrices $Q_w$ with $Q_w^*=Q_{w^*}$. By 
  hypothesis $Q_{\emptyset}=q(0)=I_\ell$, the $\ell\times \ell$ identity.

  Given a vector $\gamma \in\mathbb R^\ell$, the scalar-valued polynomial
\[
  q_{\gamma} = \sum \langle Q_w \gamma,\gamma\rangle w
\]
  is concave.  
  By the main result in \cite{HM0}, $q_{\gamma}$
  has degree at most two.    Thus, $Q_w=0$ whenever
  $w$ has length three or more. Hence, there is
  a linear pencil $\Lambda$ and a polynomial $\Sigma$
  homogeneous of degree two such that 
\[
  q = I -\Lambda - \Sigma.
\]

  Let $\Sigma_{i,j}= \Sigma_{x_i x_j}$. 
  From the concavity hypothesis, for any $n,$  pair
  $X,Y\in \gtupn,$ and $0\le t\le 1$, 
\[
 \begin{split}
  0 & \preceq 
      + \sum \Sigma_{i,j} \otimes \big( t^2 X_i X_j +t(1-t) (X_iY_j + Y_i X_j)
              + (1-t)^2 Y_i Y_j \big) \\
     & \quad - t \sum \Sigma_{i,j}\otimes X_i X_j 
           - (1-t) \sum \Sigma_{i,j}\otimes Y_i,Y_j \\
   & = t(1-t) \sum \Sigma_{i,j} \otimes  (X_i-Y_i)(X_j-Y_j) \\
   & =  t(1-t) \Sigma(Z),
 \end{split}
\]
  where $Z=X-Y$. It follows that for each $Z\in \gtupn$
  we have
 $\Sigma(Z)\succeq 0$. Since 
  a nonnegative polynomial which is homogeneous of degree two 
  has the form $s^* s$, for some (not necessarily square)
  homogeneous linear
  matrix-valued $s$ (see e.g.~\cite{McC}), the conclusion follows. 
\end{proof}

%  Given vectors $\gamma,\eta \in\mathbb R^\ell$, consider the 
%  scalar-valued polynomial,
%\[
%  q_{\gamma,\eta}= \sum_{w\in\ax} \langle Q_w \gamma,\eta \rangle w.
%\]
%  By polarization,
%\[
%  q_{\gamma,\eta} = \frac12 \left ( q_{\gamma+\eta,\gamma+\eta}
%     - q_{\gamma,\gamma}-q_{\eta,\eta} \right ).
%\]

%  Now suppose $q$ is concave.  It follows, for each 
%  unit vector $\gamma\in\mathbb R^\ell$, that $q_{\gamma,\gamma}\in\R\ax$ is 
 % concave. 

%and moreover, there exist 
%  $l_{\gamma,\gamma}$ and $\lambda_{\gamma,\gamma}$ linear 
%  so that
%\begin{equation}
% \label{eq:Qpos}
%  q_{\gamma,\gamma} = 1 - l_{\gamma,\gamma} - \lambda_{\gamma,\gamma}^*
%     \lambda_{\gamma,\gamma}.
%\end{equation}
%  Note that $\lambda_{\gamma,\gamma}$ is  vector-valued, so the last 
%  term on the right hand side is a sum of squares. 

%  From polarization, we conclude that $q$ itself has degree at most two,
%  so that
%\[
%  q = I - L  - Q,
%\]
% where $L$ is symmetric linear and $Q$ is homogeneous of degree two.  
%  Moreover, from 
% \eqref{eq:Qpos}, $Q$ is positive semidefinite. 

\subsection{From linear to concave}
\label{subsec:concaveToLin}
 The following lemma reduces the proof of 
 Theorem \ref{thm:mainConcave} for $q$ concave  to
 the case of $q$ linear.

\begin{lem}
 \label{lem:lincon}
  If $M_{d+1,d}^{\mt}(q)$ has the $(2d+1)$-PosSs-property whenever 
  $q$ is  a monic linear pencil, then $M_{d+1,d}^{\mt}(q)$ has
  the $(2d+1)$-PosSs-property whenever $q$ is concave and monic.
\end{lem}

\begin{proof}
 By Proposition \ref{prop:concaveform}, it may be assumed that
 $q\in\mathbb R^{\ell\times \ell}\ax$ is described 
 by equation  \eqref{eq:quad} for some
 linear pencil $\Ll_A\in\R^{\ell\times\ell}\ax$
and linear $s\in\R^{\ell'\times\ell}\ax$. 
Let 
\[
  Q= 
\begin{bmatrix}
  I_{\ell^\prime} &  s     \\
   s^* & I - \Ll_A  
\end{bmatrix}\in\R^{(\ell+\ell')\times(\ell+\ell')}\ax_1.
\]
 Hence $Q$ is a monic linear pencil and, as is easily
checked using Schur complements,
  $\ps_q=\ps_Q$.  Thus, a given symmetric
 $p\in\R^{\mt\times\mt}\ax$ is positive semidefinite on $\ps_q$ if and only
  if it is positive semidefinite on $\ps_Q$. 

 Let $Q=L D L^*$ be the LDU 
decomposition of $Q$, that is 
$$ L= 
\begin{bmatrix}
 I &   0 \\
 s^* &   I   
\end{bmatrix}
\qquad \text{and}\qquad
D =  
\begin{bmatrix}
 I &   0 \\
 0 &   I - \Ll - s^*s   
\end{bmatrix}.
$$
 By hypothesis, $M_{d+1,d}^{\mt}(Q)$ has the $(2d+1)$-PosSs-property
 and we are to show that $M^{\mt}_{d+1,d}(q)$ does too.
 To this end suppose $p\in\R^{\mt\times \mt}\ax$ has degree
  at most $2d+1$ 
  and is positive semidefinite on $\ps_q=\ps_Q$. 
 Hence $p$ has a representation as
\[
  p=  G  + \sum_j  \begin{bmatrix} f_j^* & g_j^* \end{bmatrix}
    Q \begin{bmatrix} f_j \\ g_j \end{bmatrix},
\]
with $g_j\in\R^{\ell\times \mt}\ax_d,$ $f_j\in\R^{\ell'\times \mt}\ax_d$
  and $G\in \Sigma^{\mt}_{d+1}$
 a sum of squares of matrix-valued polynomials of degree at most $d+1$.
  Since
\[
  L^* \begin{bmatrix} f_j\\ g_j \end{bmatrix}
  =\begin{bmatrix} f_j + sg_j \\ g_j \end{bmatrix},
\]
 it follows that
\beq\label{eq:climax}
  p= G +\sum (f_j + sg_j)^* (f_j+sg_j) + \sum g_j^* (1-\Ll -s^*s) g_j.
\eeq
 Observing that $f_j+sg_j$ has degree at most $d+1$,
\eqref{eq:climax}
  shows that $p\in M_{d+1,d}^{\mt}(q)$ and completes
 the proof. 
\end{proof}

\subsection{From $M_{d+1,d}$ to $M_{d,d}$}
 \label{subsec:fromdplusonetod}
 It turns out that in the case $q$ is monic linear, 
 $M_{d+1,d}^{\mt}(q)$ has the $(2d+1)$-PosSs-property if and only
 if $M_{d,d}^{\mt}(q)$ does.

\begin{lem}
% \label{lem:fromdplusonetod}
   Suppose $q$ is a monic linear pencil. 
   If $p\in\R^{\mt\times \mt}\ax$ has degree at most  $2d+1$   
and $p\in M_{d+1,d}^{\mt}(q)$, then
   $p\in M_{d,d}^{\mt}(q)$.
\end{lem}

\begin{proof}
 If $p\in M^{\mt}_{d+1,d}(q)$ then 
\[
  p = \sum g_j^* g_j + \sum f_j^* q f_j,
\]
  for matrix-valued polynomials $g_j$ of degree at most $d+1$ and $f_j$
  of degree at most $d$.  Any degree $2d+2$ terms in
  $\sum g_j^* g_j$ appear as (positively weighted) squares and
   can not be canceled by terms in 
 $\sum f_j^* q f_j$, since the latter have degree at most $2d+1$.
  Hence each $g_j$ must have degree at most $2d$. 
\end{proof}

 By the results of Subsections \ref{subsec:concaveToLin}
 and \ref{subsec:fromdplusonetod},
 Theorem \ref{thm:mainConcave} follows from the following a priori
 weaker statement.

\begin{prop}
\label{prop:igor}
 If $q$ is a monic linear pencil,
 then   $M_{d+1,d}^{\mt}(q)$  has the $(2d+1)$-PosSs-property.
  Its $(\mk,\mt)$-test rank is no greater than   $\mt \sigma_\#(d+1)$.
\end{prop}

The proof of Proposition \ref{prop:igor} will be given
in Section \ref{sec:posproofs} below after 
subsections on positive linear functionals on matrix-valued
polynomials and on Hankel matrices and the free noncommutative moment problem.

\subsection{Positive linear functionals and the GNS construction}
\label{subsec:flathank}
 Proposition \ref{prop:gns} below,
 %a main ingredient in the proof of Theorem \ref{thm:mainConcave}
 embodies the  well known connection, through the Gelfand-Naimark-Segal (GNS)
 construction,
 between operators and
 positive linear functionals.

 Given a Hilbert space $\cX$ and 
 a positive integer $\mt$,  let $\cX^{\oplus \mt}$
 denote the orthogonal direct sum of $\cX$ with itself $\mt$ times.
Let  $A$ be a $g$-tuple of symmetric $\ell\times \ell$
 matrices, set $q=1-\Ll_A$ with $\Lambda_A$ of the form \eqref{eq:defLam},
 and abbreviate 
\begin{equation*}
% \label{eq:MI-L}
 \index{ $M^{\mt}_k := M^\mt_{k+1,k}(I - \Ll_A)$}
    M^\mt_{k+1}= M^\mt_{k+1, k}(q).
\end{equation*}

\begin{prop}
 \label{prop:gns}
  If $\la:\R^{\mt\times \mt}\ax_{2k+2}\to \R$ is a linear functional  
  which is nonnegative on $\Sigma^\mt_{k+1}$ and  positive  on
  $\Sigma^\mt_k\setminus\{0\}$, 
  then there exists a tuple $X=(X_1,\dots,X_g)$ of
  symmetric operators on a Hilbert space $\cX$ of
  dimension at most $\mt \sigma_\#(k)=\mt \dim \R\ax_k$ and a vector 
  $\ga\in \cX^{\oplus \mt}$ such that
\begin{equation}
 \label{eq:LorX}
  \la(f)= \langle f(X)\ga,\ga\rangle 
\end{equation}
  for all $f\in\R^{\mt\times \mt}\ax_{2k+1}$, where $\langle\textvisiblespace, \textvisiblespace\rangle$
  is the inner product on $\cX$.  
  Further, if
  $\la$ is nonnegative on $M^{\mt}_{k+1}$, then $X\in\ps_q$. 

  Conversely, if $X=(X_1,\dots,X_g)$ is a tuple of
  symmetric operators on a Hilbert space $\cX$ of
  dimension $N$,  
  the vector $\ga\in\cX^{\oplus \mt},$ and $k$ is a positive
  integer, then the linear
  functional  $\la:\R^{\mt\times \mt}\ax_{2k+2}\to\R$ defined by 
\[
  \la(f)= \langle f(X)\ga,\ga\rangle 
\]
  is nonnegative on $\Sigma^\mt_{k+1}$. 
  Further, if $X\in\ps_q$, then $\la$ is nonnegative also on $M^{\mt}_{k+1}$. 
\end{prop}

\begin{proof}
  First suppose that $\la:\R^{\mt\times \mt}\ax_{2k+2} \to \R$ is 
  nonnegative on $\Sigma^\mt_{k+1}$ and  positive
  on $\Sigma_k^\mt\setminus\{0\}$.   Consider the 
 symmetric bilinear form,
  defined on the vector space $K=\R^{\mt \times 1}\ax_{k+1}$
  (row vectors of length $\mt$ whose entries are polynomials
  of degree at most $k+1$)  by
\beq\label{eq:bform}
  \langle f,h\rangle = \la(h^* f).
\eeq
 From the hypotheses, this form is positive semidefinite.

 A standard use of Cauchy-Schwarz inequality shows that the set of
 null vectors
$$ 
 \cN:= \{ f \in K \mid  \langle f , f \rangle = 0 \}
$$
is a vector subspace of $K$. 
Whence one can endow the quotient
$\tilde \cX := K /\cN$ with the induced positive 
 definite bilinear form
 making it a Hilbert space.  Further,  because
 the form \eqref{eq:bform} is positive definite on 
 the subspace $\cX=\R^{\mt\times 1}\ax_k$, 
 each equivalence class in that set has a unique representative which is 
 a $\mt$-row of polynomials of degree at most $k$. 
 Hence we can consider $\cX$ as a subspace of $\tilde\cX$ with dimension
 $\mt \sigma_\#(k)$.

 Each $x_j$ determines a multiplication operator on $\cX$. 
 For $f=\begin{bmatrix} f_1 & \cdots & f_{\mt}\end{bmatrix}\in \cX$, let 
\[
  x_j f = \begin{bmatrix} x_j f_1 & \cdots & x_j f_{\mt} \end{bmatrix}
\in\tilde\cX
\]
 and define
$X_j : \cX \to \cX$ by 
\bes
%\label{eq:Xproj}
  X_j f = P x_j f, \quad f \in \cX, \; 1 \leq j \leq g,
\ees
 where $P$ is the orthogonal projection from $\tilde\cX$ onto $\cX$ 
 (which is only needed on the degree $k+1$ part of $x_jf$). 
 From the positive definiteness of the bilinear form \eqref{eq:bform}
 on $\cX$, 
one easily sees that each $X_j$ is well defined
 and
\[ 
 \langle X_j p,r \rangle  = \langle x_j p, r  \rangle = \langle
p, x_j r \rangle =  \langle
p, X_j r \rangle
\]
 for all $p,r\in \cX$.  In particular, 
 each $X_j$ is symmetric. 

 Let $\gamma \in \cX^{\oplus \mt}$ denote the vector
 whose $j$-th entry, $\gamma_j$ has the empty word
 (the monomial 1)
 in the $j$-th entry and zeros elsewhere.  
 Finally, given words $v_{s,t} \in\ax_{k+1}$ 
  and $w_{s,t}\in\ax_{k}$ for $1\le s,t \le \mt$,
  choose $f \in \R^{\mt \times \mt }\ax $ to  have $(s,t)$-entry 
  $w^*_{s,t} v_{s,t}$. 
  %let $f=(w_{s,t}^* v_{s,t})_{s,t}.$ 
  In particular, with $e_1,\dots,e_{\mt}$
  denoting the standard orthonormal basis for $\mathbb R^{\mt}$,
we have
  \[f= \sum_{s,t=1}^\mt w_{s,t}^* v_{s,t} e_s e_t^*.\] Thus, 
\[
 \begin{split}
  \langle f(X)\gamma,\gamma\rangle 
    &=  \sum \langle f_{s,t}(X) \gamma_t,\gamma_s\rangle 
     =  \sum \langle w_{s,t}^*(X)v_{s,t}(X)\gamma_t, \gamma_s \rangle 
      = \sum \langle v_{s,t}(X)\gamma_t, w_{s,t}(X) \gamma_s \rangle \\
     & = \sum \langle P (v_{s,t} e_t^*), w_{s,t} e_s^* \rangle 
      = \sum \langle v_{s,t} e_t^*, P w_{s,t} e_s^* \rangle 
      = \sum \langle v_{s,t} e_t^*, w_{s,t} e_s^* \rangle \\
     & = \sum  \la( w_{s,t}^* v_{s,t} e_s e_t^* ) 
      = \la\big(\sum (w_{s,t}^* v_{s,t} e_s e_t^*)\big) \\
     & = \la(f). 
 \end{split}
\]
  Since any $f\in\R^{\mt\times \mt} \ax_{2k+1}$ 
  can be written as a linear combination
  of words of the form $w^*v$  with
  $w\in\ax_{k+1}$ and $v\in\ax_k$ as was done above,
  equation \eqref{eq:LorX} is established.  
 
  To prove the further statement, suppose $\la$ is nonnegative
  on $M^{\mt}_{k+1}$. Given
\[
  p =\begin{bmatrix} p_1 \\ \vdots \\ p_\ell \end{bmatrix}
    \in \cX^{\oplus \ell},
\]
  note that
\begin{equation} 
 \label{eq:ruleone}
 \begin{split}
   \langle (I-\La_A(X)) p, p \rangle 
   & =  \langle p -\sum A_j P x_j p, p\rangle 
    =  \langle p- \sum A_j x_j p, p \rangle 
      =  \big\langle (I-\sum A_j x_j)p, p \big\rangle \\
   & =  \la\big(p^* (I-\La_A(x))p\big) \ge 0.
 \end{split}
\end{equation}
Hence, $q(X)= I  - \La_A(X) \succeq 0$.

  The proof of the converse  is routine  and is not
  used in the sequel. 
\end{proof}

\begin{rem}\rm
 \label{rem:modify}  
  The proof of Proposition \ref{prop:gns}
  follows somewhat the line of a similar
  result in \cite[\S2]{McC}.
  However, some subtle points are dealt with very explicitly 
  here, since they are critical to our perfect Positivstellensatz.
  One such point worth emphasizing is that we move from a functional $\la$,
  later chosen as a separating linear functional, 
  via the tuple $(X,\gamma)$,   to a new  linear functional
  $\lambda^\prime:  \R^{\mt\times \mt}\ax \to \R$ defined by
 \beq
 \label{eq:laprime}
    \lambda^\prime(f) = \langle f(X)\gamma,\gamma\rangle.
 \eeq
  Now $\lambda^\prime$ agrees with the original $\lambda$ on
  $\R^{\mt\times \mt}\ax_{2k+1}$, but they need not
  agree on monomials of degree $2k+2$. 
  
     Equation  \eqref{eq:ruleone} is the only place where we used
     that        $\La_A$ has degree one
     in the context of
$p$ having degree $k$. 
  Then  $f=p^*(I-\La_A)p$ has degree at most $2k+1$
  and hence, in the notation of Remark \ref{rem:modify},
  $\la^\prime(f)=\la(f)$. 
  The delicate gap between $2k+2$  in the  hypotheses
  and $2k+1$ in the conclusion of the theorem is
  what permits us to obtain a perfect Positivstellensatz for
  $q$ of degree 1.
 Proposition \ref{prop:gns} and the concomitant 
   careful choice of the quadratic module 
  are key ingredients in the proof of Theorem \ref{thm:mainConcave}.  
\end{rem}

\subsection{Hankel matrices and moment problems}
% \label{sec:hankel}
 This section is designed to give perspective on Proposition \ref{prop:gns}
 and does not contain results essential to the rest of the paper.
  Proposition \ref{prop:gns} can be interpreted - and proved -
  in terms of \df{flat extensions} of free noncommutative \df{Hankel matrices}.
   
 We say that 
 a linear functional on $\R^{\mt\times \mt}\ax_{2k}$ is 
 positive (nonnegative) if it is positive 
 (nonnegative) on $\Sigma_k^\mt\setminus\{0\}$.  
 If $\mu:\R^{\mt\times\mt}\ax_{2k} \to \mathbb R$ is a linear functional,
  then the function 
\[
  H:\ax_k \times \ax_k \to \mathbb R^{\mt\times\mt}, \quad
H(u,v)=\mu(v^*u)
\]
  depends only on the product $v^*u$
  and is called a free noncommutative Hankel matrix.  Further,
  $\mu$ is positive  if and only if $H$ is positive definite in the sense that 
  for  any nonzero 
$
  f:\ax_k \to \mathbb R^{\mt}
$
  we have,
\[
   \sum_{u,v}  f(v)^* H(u,v)f(u) > 0. 
\]
  The converse is also easily verified; i.e., if the $\mt\times\mt$-block matrix
  $H=(H(u,v))_{u,v\in\ax_k}$ is positive definite and its entries
  $H(u,v)$ depend only on $v^*u$, then the
  linear functional 
\[\mu:\R^{\mt\times\mt}\ax_{2k}\to\mathbb R, \quad
  \mu(E \otimes v^* u): =  \mbox{tr}(E H(u,v))
\]
  for words $u,v\in\ax_k$ and $E\in\R^{\mt \times \mt}$, is positive. 
 Furthermore, $\mu$ is nonnegative if and only if $H$ is positive semidefinite.

  In the case that the restriction $\sigma$ 
  of $\mu:\R^{\mt\times\mt}\ax_{2k+1} \to \mathbb R$ 
  to $\R^{\mt\times\mt}\ax_{2k} \to \mathbb R$ is positive definite, 
  it is easy to check
  that there is a
  positive definite 
   $\la:\mathbb R^{\mt\times \mt}\ax_{2k+2}\to\mathbb R$
  which extends $\mu$.  The tuple $X$ and vector $\ga$ 
  in $\cX$ generated by Proposition
  \ref{prop:gns} then determine a nonnegative
  $\la^\prime:\mathbb R^{\mt\times \mt}\ax \to \mathbb R$ 
  and Hankel matrix 
  defined by
\[
  \cH(u,v)= \la^\prime(v^*u) = \langle v^*u(X)\gamma,\gamma\rangle.
\]
  Further, this extension is \df{flat} in the sense that
  the rank of (the matrix of) $\cH$ is the same as that of the Hankel
  determined by $\sigma$
  and of course $\lambda^\prime$ restricted to 
  $\R^{\mt\times\mt}\ax_{2k} \to \mathbb R$ is $\mu$.

  Finally, this process solves  a
  noncommutative moment problem.  Here the view
  is that $H=(H(u,v))_{u,v\in\ax_k}$ is a given
  positive definite Hankel matrix in which case the
  construction just described produces an infinite positive semidefinite
  Hankel matrix $\cH$ extending $H$.

  The connection between linear functionals and
  Hankel matrices in this context parallels the commutative
  case, cf.~\cite{CF,CF2,Las,Lau}, and was exploited
  in \cite{McC} where it was used to represent a given 
  positive definite (noncommutative) Hankel $H$ 
  indexed by $\ax_k$ with a tuple $X$. Indeed there
  the tuple $X$ is constructed by choosing some 
  flat extension $\tilde{H}$ of $H$ to the index set
  $\ax_{k+1}$ and then constructing the tuple $X$ 
  along the lines of the proof of Proposition \ref{prop:gns}. 

A treatment of free noncommutative Hankel matrices is also presented in
\cite{Pop}.  There
   the existence of flat extensions, 
   with necessary hypothesis, of
   noncommutative Hankel matrices which are merely
   positive semidefinite, rather than positive definite is established.  
   This article also contains generalizations of the notions
   of flat extensions to path algebras and connects flat extensions
   to sums of squares.  
  
%For more about the free noncommutative moment problems we refer
%to \cite{McC,Pop}.

\section{Proof of Theorem \ref{thm:mainConcave}}
\label{sec:posproofs}

 As explained above in Subsection \ref{subsec:fromdplusonetod} the proof
 of Theorem \ref{thm:mainConcave} will be finished once we prove
 its weaker variant, 
 Proposition \ref{prop:igor}.
 Thus, throughout $q=I-\Ll_A$ and $d$ are fixed,
  $\del=d+1$, and $\ell$ is the size of $A$; i.e., $A$ is 
  a $g$-tuple of symmetric $\ell\times \ell$ matrices.
  Recall that $M_{\al,\be}^\nu = M_{\al,\be}^\nu(I-\Ll_A)$
  is defined in equation \eqref{eq:Malbeta}. 

\subsection{The truncated quadratic module is closed}
 Recall, given a natural
 number $k$,   $\R\ax_{k}$ is  the vector space of polynomials of degree
 at most $k$ and its dimension is  denoted
by
  $\sigma_\#(k)$.  
 Fix positive integers $\alpha,\beta$ and let
  $\mk=\max\{2\al, 2\beta +1\}$. In particular,
  the quadratic module $M^{\mt}_{\al,\beta}$ of equation \eqref{eq:Malbeta}
  is a cone in $\R^{\mt\times \mt}\ax_{\mk}$ 
  (recall the degree of
  $q=I-\Ll_A$ is one).

 Given $\eps>0$, let 
\bes
\cB_\eps(n): =  \big\{ X\in\bbS_n^g \mid \|X\|\le \eps \big\} \quad
\text{and}
\quad
   \cB_\eps:= 
      \bigcup_{n\in\N} \cB_\eps(n).
\ees
 There is an
 $\eps>0$ such that for all $n\in\N$, if $X\in\bbS_n^g$ and $\|X\|\le \eps$, 
then  $I_{\ell n}-\Ll_A(X)
 \succeq \frac12.$ In particular, $\cB_\eps \subset \ps_{I-\Ll_A}$.  
 Using this $\eps$ we norm $\R^{\ell\times \mt}\ax_{\mk}$ by
\beq\label{eq:normMe}
  \|p\|:= \max\big\{\|p(X)\| \mid X\in\cB_\eps \big\}.
\eeq
(Let us point out that on the right-hand side of \eqref{eq:normMe}
the maximum is attained. This follows from the fact that
the bounded nc semialgebraic set $\cB_\eps$ is convex. We refer to \cite[Section 2.3]{HM} for details).
  Note that if  $f\in\R^{\ell\times \mt}\ax_{\beta}$ and if $\| f^* (1-\Ll_A(x))f\|\le
  N^2$,  then $\|f^* f\|\le 2N^2$. 

\begin{prop}\label{prop:qmclosed}
  The truncated quadratic module 
   $M^{\mt}_{\al,\beta} \subseteq\R^{\mt\times\mt}\ax_{\mk}$ is closed. 
%   Here, $\mk= \max\{ 2\al, 2\beta + 1\}$. 
\end{prop}

\begin{proof}
  This result is a consequence of Caratheodory's theorem on convex
 hulls \cite[Theorem I.2.3]{Ba}.  
  Suppose $(p_n)$ is a sequence from $M^{\mt}_{\al,\beta}$ 
  which converges to some
  $p\in\R^{\mt\times \mt}\ax$ of degree at most $\mk$.   
  By Caratheodory's theorem, there is an $M$ 
  (at most the dimension of $\R^{\mt\times\mt}\ax_{\mk}$  plus one)  such that
  for each $n$ there exist matrix-valued polynomials 
  $r_{n,i}\in\R^{\ell\times \mt}\ax_{\al}$ and 
  $t_{n,i}\in\R^{\ell\times \mt}\ax_{\beta}$ such that
\[
   p_n = \sum_{i=1}^{M} r_{n,i}^* r_{n,i}
         + \sum_{i=1}^M t_{n,i}^* (I-\Ll_A(x))t_{n,i}.
\]
  Since $\|p_n\| \le N^2$, it follows that $\|r_{n,i}\|\le
  N$ and likewise $\|t_{n,i}^* (1-\Ll_A(x)) t_{n,i}\|\le N^2$.  
  In view of the
  remarks preceding the proposition, we obtain $\|t_{n,i}\|
  \le \sqrt{2} N$ for all $i,n$. 
  Hence for each $i$, the sequences
  $(r_{n,i})$ and $(t_{n,i})$ are bounded in $n$.  They thus have
  convergent subsequences.   
  Tracking down  these subsequential limits finishes the  proof. 
\end{proof}
 
\subsection{Existence of a positive linear functional}
Let $\del=d+1$ and write $M^{\mt}_\del=M^{\mt}_{d+1,d}$. 
We call 
a linear functional on $\R^{\mt\times \mt}\ax_{2\del}$ 
positive (nonnegative) if it is positive (nonnegative) on $\Sigma_\del^\mt\setminus\{0\}$.

\begin{lem}
 \label{technical-lemma}
    There exists a positive  linear functional
    $\lL:\R^{\mt\times \mt}\ax_{2\del}\to \mathbb R$ 
     which is nonnegative on $M^{\mt}_\del$.
\end{lem}

\begin{proof}
As above, choose $1\geq\eps>0$ satisfying $\cB_\eps\subseteq\ps_{I-\Ll_A}.$
Select a countable dense subset $X^{(1)}$, $X^{(2)},\ldots$ of $\cB_\eps(\del)$
(e.g.~all tuples of matrices in $\cB_\eps(\del)$ with rational entries), and define
$\lL:
\R^{\mt\times \mt}\ax_{2\del}\to \mathbb R
$
as follows:
\[
\lL(p) := \sum_{i=1}^\infty \frac 1{2^i} \tr \big( p(X^{(i)}) \big).
\]
Clearly, $\lL(M_\del^{\mt})\subseteq\R_{\geq0}$. We claim that $\lL$ is strictly
positive on nonzero hermitian squares in $\Sigma^{\mt}_\del$.
Let $r\in \R^{\mt\times\mt}\ax_\del$ be arbitrary. If $\lL(r^*r)=0$, then
by density, $r$ vanishes on $\cB_\eps(\del)$, and by nonexistence of
low degree polynomial identities (see e.g.~\cite{Pro,Row}), 
$r=0$.
\end{proof}

\subsection{Separation}
\label{subsec:sep}

  The final ingredient in
  the proof of Proposition \ref{prop:igor} is a Hahn-Banach separation
  argument. Accordingly, 
  let $p\in\R^{\mt\times\mt}\ax_{2d+1}$ be given with
  $p(Y)\succeq 0$ for all $Y\in \ps_q$.  We are to show  
  $p\in M^{\mt}_\del$.

 If the conclusion  
 is false, then  by Proposition \ref{prop:qmclosed} 
 and the Hahn-Banach theorem
 there is a linear functional
 $\la:\R^{\mt\times\mt}\ax_{2\del}\to\R$ that 
 is nonnegative on $M^{\mt}_{\del}$ and
 negative on $p$. 
 Adding, if necessary, a small positive multiple of the linear
 functional $\lL$ produced by Lemma \ref{technical-lemma} to $\la$, 
 we can assume that $\la$ is  positive  (not just
 nonnegative) on $\Sigma^{\mt}_\del\setminus\{0\}$, nonnegative
 on $M^{\mt}_\del,$ and still negative on $p$. 
 But now Proposition \ref{prop:gns} 
 with $k=d$
 applies:
 there is a tuple
 of symmetric matrices $X\in \ps_q$ acting on a
finite-dimensional
 Hilbert
 space $\cX$ and a vector $\gamma$  such that  
\[
 \la(f)=\langle f(X)\gamma,\gamma\rangle
\]
  for all $f \in\R^{\mt\times\mt}\ax_{2d+1}$. In particular,
  \[
\langle p(X)\gamma,\gamma \rangle = \la(p)<0,\]
  so that $p(X)$ is not positive semidefinite, 
  contradicting $p|_{\ps_q}\succeq0$ and the proof is complete. 
\qed

 This argument is like  the classical one going back to 
 Putinar \cite{Put} and its noncommutative version in \cite{HM},
 but with a consequential difference.
 Possibly the best way to view this difference
 is in terms of the separating functional
 $\la.$  What is new here amounts to modifying $\la$ 
 to produce a new separating functional $\la'$, as in \eqref{eq:laprime}.
 It is this modified 
 functional that produces  perfection.
 In other Positivstellens\"atze, e.g.~\cite{HM}, 
 the proof does not do this modification
 of $\la$ and produces a tuple $X$ of bounded selfadjoint 
 operators which may act on an infinite-dimensional,
  rather than finite-dimensional, space and which also requires 
 $p$ to be \emph{strictly} positive on the underlying nc semialgebraic set.

\section{Applications}
We conclude this paper with applications of our main result and 
the techniques used in its proof.
First, in Subsection \ref{subsec:dominate} we revisit the theme of our paper \cite{HKM}, where
we discussed how complete positivity is equivalent to LMI domination
(i.e., inclusion of LMI domains). Here we strengthen some of 
our previous results by relaxing the assumptions.
Second, in Subsection \ref{subsec:beyondCon} we give a nonconvex
variant of Theorem \ref{thm:mainConcave} which in turn extends the directional
Positivstellensatz of \cite{HMP}.

\def\pL{L^\prime}
\def\pell{\ell^\prime}
\def\pA{A^\prime}

\subsection{Complete positivity and LMI domination}
 \label{subsec:dominate} 

In this section we assume basic familiarity with completely positive maps as presented 
e.g.~in \cite{BL04,Pau02,Pis03}.

   Suppose $L$ and $\pL$ are monic linear pencils 
in $g$ variables 
   of size $\ell$ and $\pell$ respectively. 
We say that $L$ \df{dominates} $\pL$ if $\ps_L\subseteq\ps_{\pL}$., i.e.,
$\pL|_{\ps_L}\succeq0$.
This situation is algebraically characterized by our Theorem \ref{thm:mainConcave}.

\begin{cor}\label{cor:ucp}
$L$ dominates $\pL$ if and only if $\pL \in M_{0,0}^{\pell}(L)$. Equivalently,
 $L$ dominates $\pL$ if and only if 
there are matrices $V_j\in\R^{\ell\times\pell}$ 
and a positive semidefinite
$S\in\bbS_{\pell}$ satisfying
\beq\label{eq:ucp1}
\pL(x) = S+ \sum_j V_j^* L(x) V_j.
\eeq
\end{cor}

 The following proposition eliminates the need for the positive
  semidefinite $S$ in Corollary \ref{cor:ucp}
  and the (unweighted) sum of squares term in
  the representation \eqref{thmmainitem2} of
  Theorem \ref{thm:mainConcave} in the case
  that $\ps_L$ is bounded. Further, combining
 this proposition with the argument of
 Lemma \ref{lem:lincon} eliminates the need
  for the (unweighted) sum of squares term
  in \eqref{thmmainitem1} of Theorem \ref{thm:mainConcave}. 

\begin{prop}
 \label{prop:IVLV}
   If $\ps_L$ is bounded, then there are matrices 
  $W_j\in\R^{\ell\times\pell}$  such that 
\[
 I = \sum_j W_j^* L(x) W_j.
\]
\end{prop}

\begin{cor}[cf.~\protect{\cite[Theorem 1.1]{HKM}}]%\label{cor:ucp2}
Suppose $\ps_L$ is bounded. Then $L$  dominates $L'$ if and only if
there are 
matrices $V_i\in\R^{\ell\times\pell}$ satisfying
\beq\label{eq:ucp2}
L'(x) = \sum_i V_i^* L(x) V_i.
\eeq
\end{cor}

\begin{proof}
  Factoring $S$ as $S=C^*C$ gives, in the notation of
  Proposition \ref{prop:IVLV},
\[
  S= \sum_j (W_j C)^* L(x) (W_j C).
\]
  An application of Corollary \ref{cor:ucp} then completes the proof.
\end{proof}

\def\bbS{\mathbb S}

\begin{proof}[Proof of Proposition~{\rm\ref{prop:IVLV}}]
Write $L(x)=I-\sum_j^g A_j x_j$ with $A_j \in \R^{\ell \times \ell}$.
To show there are finitely many, say $m$,
   nonzero vectors $h_k$ such that $\sum_k \langle h_k, h_k \rangle =1$ and
\[ 
  \sum_{k=1}^m \langle A_j h_k,h_k \rangle =0
\]
 for each $j,$ 
 let $\bbS^\ell$ denote the unit sphere in $\R^\ell$ and  
 consider the mapping 
\[
\bbS^\ell \to \R^g, \qquad  h \mapsto (\langle A_j h,h \rangle)_j
= \begin{bmatrix} \langle A_1 h,h \rangle & \cdots &
\langle A_g h,h \rangle \end{bmatrix}^*.
\]
If $0$ is not in the convex hull of the range of this map, then by
the Hahn-Banach theorem  there is a
linear functional $\lambda:\R^g\to\R$ such that 
\[\lambda \big( (\langle A_j h,h \rangle)_j \big) > 0
\]
for all  $h$. Let $\lambda_j=\lambda(e_j)$, where $e_1,\ldots,e_g$ is 
the standard orthonormal basis for $\mathbb R^{g}$. Then
\[
L( t \lambda_1,\ldots, t \lambda_g ) = 
I - t \sum_j \lambda_j A_j
\]
satisfies
\[
\langle L( t \lambda_1,\ldots, t \lambda_g) h,h\rangle
= \langle h,h\rangle - t \sum_j \lambda_j \langle A_jh,h\rangle 
>0
\]
 for all $t\leq 0$ and all nonzero $h$, contradicting the boundedness of
$\ps_L$.   Hence, $0$ is in the convex hull which says that the desired
$h_k$ exist.

 To complete the proof, let $V_{k,s} = h_k e_s^*$, where
 $e_1,\dots,e_{\pell}$ is the standard orthonormal basis for 
  $\mathbb R^{\pell}$. Thus, $V_{k,s}$ is the $\ell\times \ell^\prime$
 matrix expressed in terms of its columns as
\[
  V_{k,s} = \begin{bmatrix} 0 & \cdots &  0 &  h_k &  0 & \cdots 0\end{bmatrix}
\] 
 (where the $h_k$ is in the $s$-th column).  Now,
\bes
\begin{aligned}
  \sum_{k,s} V_{k,s}^* L(x) V_{k,s} 
   & = \sum_{k,s} e_s h_k^* (I-\sum_j A_j x_j) h_k e_s^* \\
   & =  \sum_s \Big(\sum_k \langle h_k,h_k\rangle 
        - \sum_k \big(\sum_j \langle A_j h_k,h_k \rangle\big) \Big) e_s e_s^* \\
   & =  \sum_s e_s e_s^* = I,
\end{aligned}
\ees
as desired.
\end{proof}

\begin{rem}\rm
Suppose $L$ dominates $L'$.
In case $\ps_L$ is not bounded, the positive $S$ in 
a certificate of the form \eqref{eq:ucp1}
\emph{is} needed in general.
An expression of the form \eqref{eq:ucp2} 
can be achieved for every $L'$ dominated by $L$ if and only if 
such a representation exists for $L'=I$. 
As seen in the proof of Proposition \ref{prop:IVLV},
this is the case if and only if there are vectors $h_k$, not all zero,
satisfying
\[ \sum_k \langle A_j h_k,h_k \rangle =0\]
 for each $j$.
By an old result of Bohnenblust (see \cite{Bon} for the original reference 
or \cite[\S 2.2]{KS} for an easier proof of a weaker statement sufficient
for our purpose), this 
happens if and only if 
$\mbox{span}(\{A_1,\dots,A_g\})$ does not contain a positive definite matrix.
\end{rem}

  Writing $L= I - \sum A_j x_j$ and $\pL = I - \sum \pA_j x_j,$ let
\[
 \mathcal S =\mbox{span}(\{I,A_1,\dots,A_g\}) \subset \bbS_{\ell}
\]
be the \df{operator system} associated to the monic linear pencil $L$, 
 and similarly for $\cS^\prime$. 
  The approach taken in \cite{HKM} was to view the inclusion
 $\ps_{L} \subset \ps_{\pL}$ (under the assumption of \emph{boundedness} 
of $\ps_L$) as saying that the unital mapping
\[
 \tau: \mathcal S \to \mathcal S^\prime
\]
  defined by $\tau(A_j)=\pA_j$ is (well-defined) completely positive and then
  applying the Arveson-Stinespring representation theorem \cite{BL04, Pau02,Pis03}
for completely positive
  maps.  
  Since the approach in this paper avoids the complete positivity machinery,
  it is interesting to note that
Theorem \ref{thm:mainConcave}
    implies both the Arveson Extension Theorem
and the Stinespring Theorem for matrices   (as opposed to operators).
To see why, suppose
   $\mathcal S$ and $\mathcal S^\prime$ are unital subspaces
   of $\bbS_{\ell}$ and $\bbS_{\pell}$
   respectively, and  $\tau:\cS \to \cS^\prime$ is
unital and   completely positive.
Choose $A_1,\dots,A_g$
   such that $\{I,A_1,\dots,A_g\}$ is a basis for $\cS$.
   By \cite[Proposition 4.3.2]{KS} the matrices
    $A_j$ can be chosen to make $\ps_L$ bounded; here $L$
denotes    the pencil 
   $I-\sum A_j x_j$.
   With $A_j^\prime = \tau(A_j)$, the pencil
   $L$ dominates the pencil $L^\prime=I-\sum \pA_j x_j$.
Now invoke Theorem \ref{thm:mainConcave} (for bounded domains) to
get Arveson's extension as well as Stinespring's theorem.
   The non-uniqueness of this construction is described
   by simultaneous invertible linear change of variables (on
   both the domain $\ps_L$ and codomain $\ps_{L'}$).

\subsection{Beyond convexity: a harsher positivity test}
\label{subsec:beyondCon}

 \def\be{\beta}
 
 The Positivstellensatz in \cite{HM}
 assumes the underlying semialgebraic set is bounded,
 whereas Theorem \ref{thm:mainConcave} assumes the
 set is convex. In this section we consider a case
 which lies in between.  
 For simplicity we take our polynomials to be scalar-valued.
%The interested reader should have no problems extending the results
%to matrix-valued polynomials. 

 Given a finite set $S$ of
 symmetric noncommutative polynomials whose degrees are at most $a,$
 let $Q=\{1-s^*s \mid s\in S\}$.  
 We will develop a positivity condition for a 
 polynomial $p$ of degree at most $2 d$ equivalent to $p$ lying in
  the convex cone
\[
  M_{d+a,\be }(Q) = \Sigma_{d+a} + 
\Big\{ \sum_{q\in Q} \sum_j^{\rm finite} f_{j,q}^* q f_{j,q} \mid 
   f_{j,q}\in\R\ax_\be \Big\}.
\]
(Here, and in the rest of this subsection, we omit the superscripts
in the notation for quadratic modules, since we are dealing only with
scalar-valued polynomials.)

 Let $\cX$ be a finite-dimensional  Hilbert space. 
 Given a vector $\zeta\in\cX$, natural number $\eta,$  and 
 a tuple $X$ of symmetric operators on $\cX$, 
 let $O_{X,\ze}^{\eta}$ denote the subspace
 $$O_{X,\ze}^{\eta}:= \{ f(X)\ze \mid f\in\R\ax_\eta \}$$
 of $\cX$ 
  and $P_{X,\ze}^{\eta}$ be the orthogonal  projection of
  $\cX$ onto this space.
 Generically, the dimension of $O_{X,\ze}^\eta$ is $\sigma_\#(\eta)$.
 The following is a free nonconvex Positivstellensatz with degree
  bounds.  

\begin{thm}[Beyond convex]
\label{thm:highDeg}
Let $p\in\R\ax_{2d}$ be symmetric and fix 
 an integer $0 \leq \be <d$.
 Assume that $\ps_Q$ contains a nontrivial
 nc neighborhood of $0$.
  If for any Hilbert space $\cX$
 of dimension  $\sigma_\#(d+a -1),$  any
  $g$-tuple of matrices $X$ acting on $\cX$
  and vector $\ze\in\cX$, 
\[
 P_{X,\ze}^{\be} \big(1- s^*(X)s(X) \big) P_{X,\ze}^{\be} \succeq 0
\quad \textrm{for all } s\in S
\]
 implies
\[
 \langle p(X) \ze, \ze\rangle \geq 0,
\]
 then $p \in M_{d+a,\be}(Q)$.
  $($The converse is obviously true.$)$
\end{thm}
In 
other words a clean Positivstellensatz holds without
 concavity of $Q$ (the collection $S$), provided
 we test positivity of  $p$ on a sufficiently 
 large class of matrices and vectors.

\begin{rem}\rm
% \label{rem:beyondbeyond}
\mbox{}\par
\ben[\rm (1)]
\item
  If $a=1$ and $\be=d$, then generically dimension counting tells us
$O_{X,d}^{d} $ is $\cX$, and we are back in the setting
 of Theorem {\rm\ref{thm:mainConcave}}.
\item 
The condition: $\langle p(X)\ze,\ze \rangle >0$ provided 
$
 \ze^* (1- s^*(X) s(X)  ) \ze \geq 0
$
is a  condition converted to a Positivstellensatz in 
 \cite{HMP}.  The $\be=0$ case of Theorem {\rm\ref{thm:highDeg}}
 improves this, indeed makes a perfect version.
\een
\end{rem}

\begin{proof}[Sketch of proof of Theorem {\rm\ref{thm:highDeg}}]
Abbreviate $M_{d+a,\be}(Q)$   to $M_{d+a,\be}$.
  Suppose $p$ has degree at most $2d$, but is not in $M_{d+a,\be}$.
  The  Proposition  \ref{prop:qmclosed} extends to show $M_{d+a,\be}$
  is closed, with an easy generalization of the same argument.
  Then there is a positive linear functional $\la:\R\ax_{2(d+a)}\to \mathbb R$
  that is nonnegative on $M_{d+a,\be}$
   but such that $\la(p)<0$;  see
  Lemma \ref{technical-lemma}, a variant of which
  is needed to see that such an $\la$ can be chosen 
  positive, not just nonnegative on $\Sigma_{d+a}\setminus\{0\}$. 
  Applying Proposition \ref{prop:gns}
  produces a finite-dimensional Hilbert space
  $\cX$, a tuple of matrices 
  $X$ on $\cX$ and cyclic vector $\gamma$ such that for 
  any polynomial $f$ of degree at most $2(d+a)-1$, 
\[
  \langle f(X)\gamma, \gamma \rangle 
      = \la(f).
\]
  In this context, the analog of the further part of Proposition 
  \ref{prop:gns} is the following. 
  If $f$ is of degree at most $d-1$ and $s\in S$, then
\[
  \big\langle (I-s(X)^*s(X)) f(X)\gamma, f(X)\gamma \big\rangle
    = \la(f^*(I-ss^*)f) \ge 0.
\]
 On the other hand, 
\[
   \langle p(X)\gamma,\gamma \rangle = \la(p)<0,
\]
yielding 
a contradiction.
\end{proof}

\end{document}